\documentclass[letterpaper]{amsart} 

\usepackage{amsmath}
\usepackage[all]{xy}

\usepackage{psfrag}
\usepackage{epsfig}
 
\numberwithin{equation}{section}

\newtheorem{theorem}{Theorem}[section] 
\newtheorem{proposition}[theorem]{Proposition}

\newtheorem{lemma}[theorem]{Lemma} 
 
\theoremstyle{definition}

\newtheorem{definition}[theorem]{Definition}

\newcommand{\credit}[1]{\smallskip\noindent {\textbf{#1.\ }}}

\def\ZZ{\mathbb{Z}}

\begin{document}

%\title[Mutation Classes Associated With Generalized Cartan Matrices]{Mutation Classes Associated With The Generalized Cartan Matrices of Size $3$}
\title{Mutation Classes of Skew-symmetrizable $3\times 3$ Matrices}

\author{Ahmet I. Seven}

\address{Middle East Technical University, Department of Mathematics, 06800, Ankara, Turkey}
\email{aseven@metu.edu.tr}

\thanks{The author's research was supported in part by the Turkish Research Council (TUBITAK)}

%\thanks{The author's research was supported in part by the Scientific and Technological Research Council of Turkey (TUBITAK) grant \# 110T207}

%\date{August 26, 2011}
\date{\today}

\begin{abstract}
Mutation of skew-symmetrizable matrices is a fundamental operation that first arised in Fomin-Zelevinsky's theory of cluster algebras; it also appears naturally in many different areas of mathematics. %???. Mutation can also be naturally viewed as an operation on certain graphs, called diagrams. 
%Mutation operation defines an equivalence relation on skew-symmetrizable matrices and their graphs, which is called mutation-equivalence. 
In this paper, we study mutation classes of skew-symmetrizable $3\times 3$ matrices and associated graphs. We determine representatives for these classes using a natural minimality condition, generalizing and strengthening results of Beineke-Brustle-Hille and Felikson-Shapiro-Tumarkin. Furthermore, we obtain a new numerical invariant for the mutation operation on skew-symmetrizable matrices of arbitrary size.

%In particular, we review the correspondence between mutation classes of skew-symmetrizable matrices and generalized Cartan matrices in the finite and affine type. 
%Furthermore, we give a description of the mutation classes associated with the generalized Cartan matrices of size $3$, generalizing results of Beineke-Bruestle-Hille. %In particular, we give an interpretation of the Markov constant in terms of generalized Cartan matrices.  
\end{abstract}

%\subjclass[2000]{Primary:
%05E15,  % Combinatorial problems concerning the classical groups
%Secondary:
%05C50, % Graphs and matrices
%15A36, % Matrices of integers,
%17B67. % Kac-Moody (super)algebras (structure and representation theory)
%}
\maketitle
%\tableofcontents

\section{Introduction}
\label{sec:intro}

Mutation of skew-symmetrizable matrices is a fundamental operation that first arised in Fomin-Zelevinsky's theory of cluster algebras; it also appears naturally in many different areas of mathematics. Mutation can also be naturally viewed as an operation on certain graphs, called diagrams. 
%Mutation operation defines an equivalence relation on skew-symmetrizable matrices and their graphs, which is called mutation-equivalence. 
In this paper, we study mutation classes of skew-symmetrizable $3\times 3$ matrices and their diagrams. We determine representatives for these classes using a natural minimality condition, generalizing and strengthening results of \cite{BBH,FSTu}. Furthermore, we obtain a new numerical invariant for the mutation operation on skew-symmetrizable matrices of arbitrary size.

%One of the recent developments in representation theory has been the introduction of cluster algebras by Fomin and Zelevinsky to provide an algebraic framework for a  study of Lusztig's canonical bases and positivity in algebraic groups. It is now well-known that these algebras are also closely related with many different areas of mathematics (see, e.g., \cite{Ke} for an account of these connections). A particular analogy exists between combinatorial aspects of cluster algebras and Kac-Moody algebras: roughly speaking, cluster algebras are associated with skew-symmetrizable matrices while Kac-Moody algebras correspond to (symmetrizable) generalized Cartan matrices. 
%The goal of this paper is to describe an interplay between these two classes of matrices in size $3$, characterizing the skew-symmetrizable matrices associated with the generalized Cartan matrices. (We will deal with combinatorial aspects of these algebras; we will not need or use their algebraic properties, including their definition.)

%The goal of this paper to describe a particular analogy between cluster algebras and Kac-Moody algebras: roughly speaking, cluster algebras are associated with skew-symmetrizable matrices while Kac-Moody algebras correspond to (symmetrizable) generalized Cartan matrices. In this paper, we describe an interplay between these two classes of matrices. 
%We will deal with the combinatorial aspects of these algebras, so we will not need their definition nor their algebraic properties.
To state our results, we need some terminology. Let us recall that an integer matrix $B$ is skew-symmetrizable if $DB$ is skew-symmetric for some diagonal matrix $D$ with positive diagonal entries. 
%One of the main inventions in the theory of cluster algebras is an explicitly-defined, yet mysterious, operation, called mutation, on skew-symmetrizable matrices. More precisely, 
For any matrix index $k$, the mutation of a skew-symmetrizable matrix $B$ at $k$ is another skew-symmetrizable matrix $\mu_k(B)=B'$: 
\begin{displaymath}
B' = \left\{ \begin{array}{ll}
%E_{k,k}=-1 & \\
B'_{i,j}=-B_{i,j} & \textrm{if $i=k$ or $j=k$}\\
%E_{k,j}=max\{0,-B_{k,j}\} & \textrm{if $j\ne k$}\\
B'_{i,j}=B_{i,j}+sgn(B_{i,k})[B_{i,k}B_{k,j}]_+ & \textrm{else}
\end{array} \right.
\end{displaymath}
(where we use the notation $[x]_+=max\{x,0\}$ and $sgn(x)=x/|x|$ with $sgn(0)=0$). 
Mutation is an involutive operation, so repeated mutations give rise to the \emph{mutation-equivalence} relation on skew-symmetrizable matrices. 

%; each mutation-equivalence class uniquely determines, in particular, a cluster algebra. Therefore it is natural to ask for an explicit description and classification of these mutation classes. 

On the other hand, %Let us also recall a related combinatorial construction from \cite{CAII}: 
motivated by the Dynkin diagram construction in the thory of Kac-Moody algebras \cite{K}, 
for any skew-symmetrizable $n\times n$ matrix $B$, a directed graph $\Gamma (B)$, called diagram of $B$, is associated in \cite{CAII} as follows: the vertices of $\Gamma (B)$ are the indices $1,2,...,n$ such that there is a directed edge from $i$ to $j$ if and only if $B_{ij} > 0$, and this edge is assigned the weight $|B_{ij}B_{ji}|\,$. 
Let us note that if $B$ is not skew-symmetric, then the diagram $\Gamma(B)$ does not determine $B$ as there could be several different skew-symmetrizable matrices whose diagrams are equal; however, if a skew-symmetrizing matrix $D$ is fixed, then $\Gamma(B)$ determines $B$. In any case, we use the general term diagram to mean the diagram of a skew-symmetrizable matrix. Then the mutation $\mu_k$ can naturally be viewed as a transformation on diagrams (see Section~\ref{sec:pre}
for a description). In the particular case where the vertex $k$ is a source (resp. sink), i.e. all incident edges are oriented away (resp. towards) $k$, then $\mu_k$ acts by only reversing all edges incident to $k$; in that case we also call $\mu_k$ a reflection (as in classical Bernstein-Gelfand-Ponomarev reflection functors). 
Note also that if $B$ is skew-symmetric then the diagram $\Gamma(B)$ may be viewed as a quiver and the corresponding mutation operation is also called quiver mutation . There are several categorical interpretations of the quiver mutation, we refer to \cite{Ke} for a survey.

%We also use the following terminology: a diagram $\Gamma$ is mutation-cyclic if any diagram which is mutation-equivalent to $\Gamma$ is cyclic, otherwise we call it mutation-acyclic

%???$\Gamma$ is mutation-cyclic if any diagram which is mutation-equivalent to $\Gamma$ is cyclic, otherwise we call it mutation-acyclic

%This operation of mutation is  defined by an explicit, yet mysterious, piecewise-linear formula (Definition~\ref{def:mut-skew-symmetrizable}). It can also be viewed more combinatorially as an operation on graphs (Section~\ref{sec:skew-mut}). It also has several category theoretical interpretations \cite{Ke}. Mutation defines an equivalence relation on skew-symmetrizable matrices; each mutation-equivalence class uniquely determines, in particular, a cluster algebra. Therefore it is natural to ask for an explicit description and classification of these mutation classes. 

Given the appearance of the mutation operation in many different areas of mathematics, it is natural to study properties of the mutation classes of skew-symmetrizable matrices and the associated diagrams. Currently, a description of these classes are known for finite and affine types \cite{BGZ,S3}, there is also a classification for the so-called finite mutation type diagrams \cite{FSTu2}. In this paper, we consider the next basic case of size $3$ skew-symmetrizable matrices, which is crucial to understand the mutation operation in general size. To be able to state our results,
%We generalize and strenghten results ??? for the special case of skew-symmetric matrices (equivalently quivers). 
let us recall a little bit more terminology.
By a {subdiagram} of $\Gamma$, we always mean a diagram obtained from $\Gamma$ by taking an induced (full) directed subgraph on a subset of vertices and keeping all its edge weights the same as in $\Gamma$. By a cycle we mean a subdiagram whose vertices can be labeled by elements of $\ZZ/m\ZZ$ so that the edges betweeen them are precisely $\{i,i+1\}$ for $i \in  \ZZ/m\ZZ$. We call a diagram $\Gamma$ \emph{mutation-acyclic} if it is mutation-equivalent to an acyclic diagram (i.e. a diagram which has no oriented cycles at all); otherwise we call it \emph{mutation-cyclic}. Now we can state our first main result:

\begin{theorem}\label{th:minimum}
Suppose that $M$ is a mutation class of diagrams with $3$ vertices. For any $\Gamma$ in $M$, let $s(\Gamma)$ denote the sum of the square roots of the weights in $\Gamma$. Then there is a diagram $\Gamma_0$ in $M$ such that $s(\Gamma_0)$ is minimal.
Furthermore, we have the following:
\begin{enumerate}

\item[(i)]
If $M$ is a mutation class of mutation-cyclic diagrams, then $\Gamma_0$ is unique upto a change of orientation which reverses all edges (and upto an enumeration of vertices). 
%change of orientation, such that $s(\Gamma_0)$ is minimal. (???if mutation-cyclic then $\Gamma_0$ is unique???)

%The diagram $\Gamma_0$ is unique up to a change of orientation.
\item[(ii)]
If $M$ is the mutation class of an acyclic diagram, then $\Gamma_0$ is acyclic and it is unique upto a reflection at a source or sink (and upto an enumeration of vertices).

\end{enumerate}

\end{theorem}

\noindent
Note that in this theorem part (ii) generalizes and strengthens \cite[Theorem~9.1]{FSTu} (which claims uniqueness upto an arbitrary change of orientation for quivers). In fact, part (ii) establishes a special case (rank three) of a standard conjecture of cluster algebra theory \cite[Conjecture~4.14 (4)]{FZ-CDM03}, which states that mutation-equivalent acyclic diagrams can be obtained from each other by a sequence of reflections at sources or sinks. For quivers, this conjecture was obtained in \cite[Corollary~4]{CK} using categorical methods. We use more elementary algebraic-combinatorial methods. Let us also note that a numerical criterion to check whether a given diagram with three vertices is mutation-acyclic has been obtained by the author in \cite{S5}. (This criterion is recalled in Theorem~\ref{th:acyclic}).
%(Recall that in this paper we consider uniqueness of diagrams upto an enumeration of vertices.)
%Note this implies in particular that any two 

We also characterize $\Gamma_0$ using a ``local" property, generalizing \cite[Lemma~2.1]{BBH} and \cite[Theorem~9.1(3)]{FSTu}):

\begin{theorem}\label{th:minimum2}
Suppose that $M$ is a mutation class of diagrams with $3$ vertices. Let $\Gamma_0$ be the diagram in $M$ such that $s(\Gamma_0)$ is minimal as in the Theorem~\ref{th:minimum}. Then we have the following:
%Then $\Gamma_0$ is the unique, with the same properties as in Theorem~\ref{th:minimum}, diagram in $M$ such that, for each vertex $i$, we have $s(\Gamma_0)\leq s(\mu_i(\Gamma_0$.

%Furthermore, for each $\Gamma$ in $M$, there is a sequence $\{\mu_i\}$ of mutations with $\Gamma_0=\mu_n...\mu_1(\Gamma)$ such that for $\Gamma_i=\mu_i...\mu_1(\Gamma)$ we have $s(\Gamma_i)<s(\Gamma_{i+1})$. Furthermore, if $\Gamma$ is mutation-cyclic then the sequence $\{\mu_i\}$ is uniquely determined: specifically, the vertex $i$ is the vertex which is not incident to the edge with maximal weight in $\Gamma_{i-1}$,(here when $i=1$, it is assumed to be vertex in $\Gamma$ ). Conversely, for any \emph{maximal} sequence $\{\mu_i\}$ such that $s(\Gamma_i)<s(\Gamma_{i+1})$ with $\Gamma_i=\mu_i...\mu_1(\Gamma)$, we have $\Gamma_0=\mu_1...\mu_n(\Gamma)$. ???orientation???

\begin{enumerate}

\item[(i)]
$\Gamma_0$ is the unique, upto the same conditions as in Theorem~\ref{th:minimum}, diagram in $M$ such that, for each vertex $i$, we have $s(\Gamma_0)\leq s(\mu_i(\Gamma_0))$.

\item[(ii)]
For each $\Gamma$ in $M$, there is a (possibly empty) sequence $\{\mu_i\}$ of mutations with $\Gamma_0=\mu_1...\mu_n(\Gamma)$ such that for $\Gamma_{i-1}=\mu_i...\mu_n(\Gamma)$ we have $s(\Gamma_{i-1})<s(\Gamma_{i})$, here $i=1,...,n$ with $\Gamma_{n}=\Gamma$. Furthermore, if $\Gamma$ is mutation-cyclic then the sequence $\{\mu_i\}$ is uniquely determined: specifically, the vertex $i$ is the vertex which is not incident to the edge with maximal \footnote{We will show that $\Gamma_{i+1}$ has a unique edge with maximal weight (Lemma~\ref{lem:mut-cyc-non-min}). } weight in $\Gamma_{i+1}$, $i=0,1,...,n-1$.  %???orientation???

% when $i=1$, it is assumed to be vertex in $\Gamma$ )???orientation???

Conversely, for any \emph{maximal} sequence $\{\mu_i:i=1,...,n\}$ such that $s(\Gamma_{i-1})<s(\Gamma_{i})$ with $\Gamma_{i-1}=\mu_i...\mu_n(\Gamma)$ and $\Gamma_n=\Gamma$, we have $\Gamma_0=\mu_1...\mu_n(\Gamma)$. 

\end{enumerate}

\end{theorem}

We also obtain the following result which gives a new numerical invariant for the mutation of diagrams with any number of vertices.

\begin{theorem}\label{th:gcd}
Suppose that $\Gamma$ is a diagram with n vertices. For any vertex $i$ in $\Gamma$, let $\delta_i=\delta_i(\Gamma)$ be the greatest common divisor of the weights of the edges which are incident to $i$. Let $\delta(\Gamma)=(\delta_1,\delta_2,...,\delta_n)$ be the ordered sequence of these greatest common divisors such that $\delta_1\geq \delta_2 \geq... \geq \delta_n$. Then for any $\Gamma'$ which is mutation-equivalent to $\Gamma$, we have $\delta(\Gamma)=\delta(\Gamma')$. 

\end{theorem}

\noindent
Note that if $\Gamma$ is the diagram of a skew-symmetric matrix, then the same conlusion holds if $\delta_i$ is defined as the greatest common divisor of the radicals of the weights of the edges which are incident to the vertex $i$. (Equivalently, in the quiver notation that represents skew-symmetric matrices, the conlusion of the theorem holds if $\delta_i$ is defined as the greatest common divisor of the number of arrows in the edges which are incident to the vertex $i$.)

We prove our results in Section~\ref{sec:proof} after some preparation in Section~\ref{sec:pre}.

\section{Preliminaries}
\label{sec:pre}

In this section, we will recall some more terminology and prove some statements that we will use to prove our results. First, let us recall that the diagram of a skew-symmetrizable (integer) matrix has the following property: 

\begin{align}
\label{eq:perfect-sq}
&\text{the product of weights along any cycle is a perfect square, i.e. the square}
\\
\nonumber
&\text{of an integer. }
\end{align}

\noindent
Thus we can use the term diagram to mean a directed graph, with no loops or two-cycles, such that the edges are weighted with positive integers satisfying \eqref{eq:perfect-sq}. Let us note that if an edge in a diagram has weight equal to one, then we do not specify its weight in the picture. 
%Let us also note that we will usually consider equality of diagrams upto an enumeration of vertices.

For any vertex $k$ in a diagram $\Gamma$, the associated mutation $\mu_k$ changes $\Gamma$ as follows \cite{CAII}:
\begin{itemize} 
\item The orientations of all edges incident to~$k$ are reversed, 
their weights intact. 
\item 
For any vertices $i$ and $j$ which are connected in 
$\Gamma$ via a two-edge oriented path going through~$k$ (see  
Figure~\ref{fig:diagram-mutation-general}), 
the direction of the edge $\{i,j\}$ in $\mu_k(\Gamma)$ and its weight $\gamma'$ are uniquely determined by the rule 
%the diagram mutation affects the edge connecting $i$ and $j$ in the way shown
%, where the weights $c$ and $c'$ are 
%related by 
\begin{equation} 
\label{eq:weight-relation-general} 
\pm\sqrt {\gamma} \pm\sqrt {\gamma'} = \sqrt {\alpha\beta} \,, 
\end{equation} 
where the sign before $\sqrt {\gamma}$ 
(resp., before $\sqrt {\gamma'}$) 
is ``$+$'' if $i,j,k$ form an oriented cycle 
in~$\Gamma$ (resp., in~$\mu_k(\Gamma)$), and is ``$-$'' otherwise. 
Here either $\gamma$ or $\gamma'$ can be equal to~$0$, which means that the corresponding edge is absent. 
 
\item 
The rest of the edges and their weights in $\Gamma$ 
remain unchanged. 
\end{itemize} 

\begin{figure}[ht] 
\begin{center}
\setlength{\unitlength}{1.5pt} 
\begin{picture}(30,17)(-5,0) 
\put(0,0){\line(1,0){20}} 
%\put(20,0){\vector(-1,0){12}} 
\put(0,0){\line(2,3){10}} 
\put(0,0){\vector(2,3){6}} 
\put(10,15){\line(2,-3){10}} 
\put(10,15){\vector(2,-3){6}} 
\put(0,0){\circle*{2}} 
\put(20,0){\circle*{2}} 
\put(10,15){\circle*{2}} 
\put(2,10){\makebox(0,0){$\alpha$}} 
\put(18,10){\makebox(0,0){$\beta$}} 
\put(10,-4){\makebox(0,0){$\gamma$}} 
\put(10,19){\makebox(0,0){$k$}} 
\end{picture} 
$ 
\begin{array}{c} 
\stackrel{\textstyle\mu_k}{\longleftrightarrow} 
\\[.3in] 
\end{array} 
$ 
\setlength{\unitlength}{1.5pt} 
\begin{picture}(30,17)(-5,0) 
\put(0,0){\line(1,0){20}} 
%\put(0,0){\vector(1,0){12}} 
\put(0,0){\line(2,3){10}} 
\put(10,15){\vector(-2,-3){6}} 
\put(10,15){\line(2,-3){10}} 
\put(20,0){\vector(-2,3){6}} 
\put(0,0){\circle*{2}} 
\put(20,0){\circle*{2}} 
\put(10,15){\circle*{2}} 
\put(2,10){\makebox(0,0){$\alpha$}} 
\put(18,10){\makebox(0,0){$\beta$}} 
\put(10,-4){\makebox(0,0){$\gamma'$}} 
\put(10,19){\makebox(0,0){$k$}} 
\end{picture} 
\end{center}
 
\vspace{-.2in} 
\caption{Diagram mutation} 
\label{fig:diagram-mutation-general} 
\end{figure}
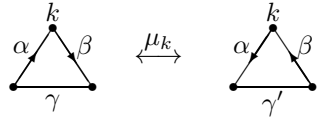

\noindent 
This operation is involutive, i.e. $\mu_k(\mu_k(\Gamma))=\Gamma$, so it defines an equivalence relation on the set of all diagrams. More precisely, two diagrams are called \emph{mutation-equivalent} if they can be obtained from each other by applying a sequence of mutations. The \emph{mutation class} of a diagram $\Gamma$ is the set of all diagrams which are mutation-equivalent to $\Gamma$. If $B$ is a skew-symmetrizable matrix, then $\Gamma(\mu_k(B))=\mu_k(\Gamma(B))$ (see 
Section~\ref{sec:intro} for the definition of $\mu_k(B)$). Let us note that if $B$ is not skew-symmetric, then the diagram $\Gamma(B)$ does not determine $B$ as there could be several different skew-symmetrizable matrices whose diagrams are equal; however, if a skew-symmetrizing matrix $D$ is fixed, then $\Gamma(B)$ determines $B$, so mutation class of $\Gamma(B)$ determines that of $B$ (the matrix $\mu_k(B)$ shares the same skew-symmetrizing matrix $D$ with $B$ \cite{CAII}).

In this paper, we will mainly consider diagrams with exactly three vertices. Therefore it will be convenient for us to use special notation for these diagrams, generalizing the one used in \cite{BBH} and \cite{FSTu}:
\begin{definition}
\label{def:(a,b,c)}
Suppose that $\Gamma$ is a three-vertex diagram with weights $\alpha,\beta$ and $\gamma$. Let $a=\sqrt{\alpha},b=\sqrt{\beta}$ and $c=\sqrt{\gamma}$. We call $a,b,c$ the \emph{radical weights} of $\Gamma$ and use the following notation: if $\Gamma$ is acyclic we write $\Gamma=(a,b,c)^-$; if $\Gamma$ is cyclic we write $\Gamma=(a,b,c)$, 
without considering any particular ordering. We denote by $s(\Gamma)$ the sum of the radical weights of $\Gamma$. By the definition of a diagram, the product of the radical weights is an integer by \eqref{eq:perfect-sq}.
\end{definition}
\noindent
Note that if $B$ is a skew-symmetric matrix then the radical weights of $\Gamma(B)$ are equal to the positive entries of $B$. Also note that this notation does not uniquely determine $\Gamma$. Nevertheless it is convenient for us because it behaves well under the mutation operation: 

\begin{proposition}\label{prop:mut-radical-weight}
Suppose that $\Gamma$ is a diagram with three vertices. Then we have the following:

\begin{enumerate}

\item[(i)]

If $\Gamma=(a,b,c)^-$ and $k$ is a vertex which is a source or sink in $\Gamma$, then $\mu_k(\Gamma)=(a,b,c)^-$.

\item[(ii)]
Suppose that $\Gamma=(a,b,c)^-$ and $k$ is a vertex which is neither a source nor sink in $\Gamma$. Also assume that $k$ is not incident to the edge with radical weight $c$. Then $\mu_k(\Gamma)=(a,b,c+ab)$

\item[(iii)]
Suppose that $\Gamma=(a,b,c)$ and $k$ is the vertex which is not incident to the edge with radical weight $c$. Then $\mu_k(\Gamma)=(a,b,ab-c)$ (resp. $\mu_k(\Gamma)=(a,b,c-ab)^-$) provided $c<ab$ (resp. $c\geq ab$).

\end{enumerate}

\end{proposition}

\noindent
The proposition follows from the definition of the mutation operation and the following technical statement, providing skew-symmetrization by conjugation:
%\footnote{\cite{PY}}???

\begin{lemma} 
\label{lem:ss-conjugation} \cite[Proposition~8.1]{CAII}
Let $B$ be a skew-symmetrizable (integer) matrix %and suppose that $D$ is a skew-symmetrizing matrix for $B$, i.e., 
%a diagonal matrix with positive diagonal entries such that $DB$ is skew-symmetric.  
Then there exists a diagonal matrix $H$ with 
positive diagonal entries such that 
$HBH^{-1}$ is skew-symmetric. 
Furthermore, the matrix $S(B)=(S_{ij})=HBH^{-1}$ is 
uniquely determined by~$B$. 
Specifically, the matrix entries of $S(B)$ are given by 
\begin{equation} 
\label{eq:s-matrix} 
S_{ij} = 
{\rm sgn}(B_{ij}) \textstyle\sqrt{|B_{ij} B_{ji}|}\,. 
\end{equation} 
\noindent
%We call $S(B)$ the \emph{canonical skew-symmetrization} of~$B$. 
% this term is never used elsewhere 
Furthermore, for any matrix index $k$, we have $S(\mu_k(B))=\mu_k(S(B))$.

%The statement also holds for symmetrizable matrices, replacing $B$ by a quasi-Cartan companion $A$ respectively. 
The matrix $H$ can be taken as $D^{1/2}$ where $D$ is a skew-symmetrizing matrix for $B$.

\end{lemma}
\noindent 
%This lemma is obtained by direct computation with $H=D^{1/2}$. For skew-symmetrizable matrices it was obtained in \cite[Proposition~8.1]{CAII}; the same proof works in the symmetrizable case as well.

Let us also record the following statements for convenience; they can be checked easily using the definition of the mutation operation:

\begin{proposition}\label{prop:weight-s}
%(LEMMA2(i,ii)) 
Suppose that $\Gamma$ is a diagram with three vertices. Then we have the following:

\begin{enumerate}

\item[(i)]
$s(\Gamma)=s(\mu_k(\Gamma))$ if and only if $\Gamma$ and $\mu_k(\Gamma)$ have the same weights; furthermore:

\begin{enumerate}

\item[(a)]
the diagrams $\Gamma$ and $\mu_k(\Gamma)$ are both cyclic or both acyclic,

\item[(b)]
if $\Gamma$ is acyclic, then the vertex $k$ is a source or sink in $\Gamma$.
\end{enumerate}

% \footnote{note then that the diagrams $\Gamma$ and $\mu_k(\Gamma)$ are both cyclic or both acyclic}. 

%??? If, furthermore, $\Gamma$ and $\mu_k(\Gamma)$ are  cyclic, then $\mu_k(\Gamma)$ can be obtained from $\Gamma$ by reversing all the edges???
\item[(ii)]
$s(\Gamma)>s(\mu_k(\Gamma))$ if and only if the edge which is not incident to the vertex $k$ has smaller weight in $\mu_k(\Gamma)$ than in $\Gamma$ (the weights of the remaining edges are equal).

\end{enumerate}

\end{proposition}

\begin{proposition}\label{prop:weight 4}
Suppose that $\Gamma$ is a three-vertex diagram which has an edge whose weight is less than $4$. Then $\Gamma$ is mutation-acyclic.

%Furthermore, if $\Gamma$ is cyclic, then there is a vertex $k$ such that $s(\Gamma)>s(\mu_k(\Gamma))$; in particular, $s(\Gamma)$ is not minimal.

\end{proposition}

\begin{proof} 
Suppose that $\Gamma=(a,b,c)$ is cyclic such that $c\leq a,b$ (so $c < 2$, thus $c=1,\sqrt{2}$ or $\sqrt{3}$). Let us also assume, without loss of generality, that $a\leq b$. Let $i$ (resp. $j$) be the vertex which is not incident to the edge with radical weight $b$ (resp. $a$).  If $c=1$, then $\mu_i(\Gamma)=(a,b-a,c)^-$ is acyclic. Let us assume now that $c=\sqrt{2}$. If $ac\leq b$, then $\mu_i(\Gamma)=(a,b-ac,c)^-$ is acyclic, otherwise $\mu_j\mu_i(\Gamma)=(bc-a,ac-b,c)^-$ is acyclic (because $bc-a>0$ for our assumption $a\leq b$ and $c=\sqrt{2}>1$). For $c=\sqrt{3}$ we use a similar argument: if $ac\leq b$, then $\mu_i(\Gamma)=(a,ac-b,c)^-$ is acyclic, otherwise either $\mu_j\mu_i(\Gamma)=(bc-2a,ac-b,c)^-$ is acyclic or (i.e. if $bc-2a<0$) $\mu_i\mu_j\mu_i(\Gamma)=(2a-bc,2b-ac,c)^-$ is acyclic (note $2b-ac>0$ because $b\geq a$ and $2>c$). This completes the proof.
%Suppose that $\Gamma=(\alpha,\beta,\gamma)$ is cyclic such that $\gamma < 4$ (so $\gamma=1,{2}$ or ${3}$).
%Then by the definition of diagram we may assume that $\beta=b \sqrt{\gamma}$, and $a=\sqrt{\alpha}$ is integer. Let $i$ be the vertex opposite to $\beta$ and $j$ opposite to $\alpha$. Note that if $a\leq b$, then $\mu_i(\Gamma)=(a,(b-a)\sqrt(\beta),\gamma)^-$ acyclic; if $b<a$, then $\mu_j(\Gamma)=(\alpha',\beta, \gamma)$ is acyclic (if $b\gamma\leq a$), otherwise it is cyclic with $\alpha'<\alpha$ so we apply induction (on $a=\sqrt{\alpha}$).
%For the second part,  suppose that $\Gamma=(a,b,c)$ with $a^2<4$ (so $a=1,\sqrt{2},\sqrt{3}$). If the conclusion of the lemma is not satisfied then we have $bc\geq 2a$, $ac\geq 2b$ and $ab\geq 2c$ (here $\mu_i(\Gamma)$ is cyclic for any vertex $i$ in $\Gamma$). Note that from the last inequality we have $b\geq 2c/a$, on the other hand, from the second inequality, we have $c\geq 2b/a$, implying $b\geq 4b/a^2$; however this is not possible because $a^2<4$. 
\end{proof}

Determining whether a given diagram is mutation-acyclic or not is a natural problem in the theory of cluster algebras and related topics. For diagrams with three vertices, a numerical criterion for being mutation-acyclic has been obtained by the author in \cite{S5}, using the notion of a \emph{quasi-Cartan companion}. For the convenience of the reader, we will recall this criterion. First let us recall that a quasi-Cartan companion of a skew-symmetrizable matrix $B$ is a symmetrizable matrix $A$ whose diagonal entries are equal to $2$ and whose off-diagonal entries differ from the corresponding entries of $B$ only by signs \cite{BGZ}. A quasi-Cartan companion $A$ of skew-symetrizable matrix $B$ is called \emph{admissible} if it satisfies the following sign condition: for any cycle $Z$ in $\Gamma(B)$, the product $\prod _{\{i,j\}\in Z}(-A_{i,j})$ over all edges of $Z$ is negative if $Z$ is oriented and positive if $Z$ is non-oriented \cite{S3}. The main examples of admissible companions are the generalized Cartan matrices: if $\Gamma(B)$ is acyclic, i.e. has no oriented cycles at all, then the quasi-Cartan companion $A$ with $A_{i,j}= -|B_{i,j}|$, for all $i\ne j$, is admissible. However, for an arbitrary skew-symmetrizable matrix $B$, an admissible quasi-Cartan companion may not exist; if exists it is unique upto simultaneous sign changes in rows and columns. For any skew-symmetrizable matrix $B$ of size $3$, an admissible quasi-Cartan companion exists and it determines whether its diagram $\Gamma(B)$ is mutation-acyclic:

\begin{theorem}\label{th:acyclic} \cite[Theorem~2.6]{S5}
Suppose that $B$ is a skew-symmetrizable matrix of size $3$ and let $A$ be an admissible quasi-Cartan companion of $B$. Then $\Gamma(B)$ is mutation-acyclic if and only if one of the following holds:
\begin{enumerate}
\item[(i)] 
$det(A)>0$ and $A$ is positive\footnote{$A$ is called (semi)positive if $DA$ is positive (semi)definite, where $D$ is a symmetrizing matrix of $A$.},
\item[(ii)] 
$det(A)=0$ and $A$ is semipositive of corank $1$,
\item[(iii)] 
$det(A)<0$.
%For the vector $u$ in part (i), the subquiver $supp_Q(u)$ has exactly two vertices or it is a cycle.
\end{enumerate}
\end{theorem}

\noindent
Let us note that parts (i) and (ii) occur if and only if $\Gamma(B)$ is mutation-equivalent to a Dynkin and an extended Dynkin diagram respectively \cite{BGZ,S3} (here a Dynkin diagram is an orientation of a Dynkin graph). Let us also mention that the main ingredient in proving the theorem is an extension of the mutation operation to quasi-Cartan companions; we refer to \cite[Section~2]{S3} for details.
%\noindent
%Let us remark that determinant of a quasi-Cartan companion, or Markov constant, does not determine the mutation class of a skew-symmetrizable matrix. 

The previous Theorem~\ref{th:acyclic} was obtained in \cite{S5} as a non-trivial generalization of a characterization in \cite{BBH} for \emph{skew-symmetric} matrices of size $3$, using a polynomial called the Markov constant. More explicitly, for a skew-symmetric $3\times 3$ matrix $B$ with $\Gamma(B)=(x,y,z)$, the associated Markov constant is defined as $C(B)=C(x,y,z)=x^2+y^2+z^2-xyz$. Then skew-symmetric matrices with mutation-acyclic diagrams can be characterized as follows:

%For \emph{skew-symmetric} matrices of size $3$, a particular characterization has been obtained in \cite{BBH} using a polynomial called the Markov constant. More explicitly, for a skew-symmetric $B$ with $\Gamma(B)=(x,y,z)$, the associated Markov constant is defined as $C(B)=C(x,y,z)=x^2+y^2+z^2-xyz$. 

%\begin{definition}
%\label{def:Markov}
%Suppose that $B$ is a $3\times 3$ skew-symmetric matrix whose diagram $\Gamma(B)$ is cyclic. Let $x,y,z$ be the positive entries of $B$ (so the weights of $\Gamma(B)$ are $x^2$,$y^2$ and $z^2$, see Figure~\ref{fig:diag-3x3}). We define the associated Markov constant as $C(B)=C(x,y,z)=x^2+y^2+z^2-xyz$. 
%\end{definition}
%\noindent
%Note that $C(B)$ is invariant under simultaneous permutations of rows and columns (i.e. it is invariant under permutations of the vertices in $\Gamma(B)$).

%\begin{figure} 
%$  \xymatrix{ & \circ \ar^{y^2}[dr] & \\ \circ  \ar^{x^2}[ur]  & & \circ \ar^{z^2}[ll]  }  $
%\caption{Diagram of a $3\times 3$ skew-symmetric matrix. (Note that quiver notation is used in \cite{BBH})} 
%\label{fig:diag-3x3} 
%\end{figure} 

%Let us also recall from \cite{BBH} how skew-symmetric matrices with mutation-acyclic diagrams are characterized by the Markov constant:

\begin{theorem}\cite[Theorem~1.1]{BBH} \label{th:hille} 
Suppose that $B$ is a skew-symmetric (integer) matrix such that $\Gamma(B)=(x,y,z)$ is cyclic. 
Then the following are equivalent: \\
(1) $\Gamma(B)$ is mutation-acyclic. \\
(2) The Markov constant satisfies $C(x,y,z) > 4$  or $\min\{x,y,z\} < 2$. \\
(3) The Markov constant satisfies $C(x,y,z) > 4$ or the triple $(x,y,z)$
is in the following list (where we assume $x \geq y \geq z$): \\
\hspace*{0.3cm} a) $C(x,y,z) = 0:$ $(x,y,z) = (0,0,0)$, \\
\hspace*{0.3cm} b) $C(x,y,z) = 1:$ $(x,y,z) = (1,0,0)$, \\
\hspace*{0.3cm} c) $C(x,y,z) = 2:$ $(x,y,z) = (1,1,0)$ or $(1,1,1)$, \\
\hspace*{0.3cm} d) $C(x,y,z) = 4:$ $(x,y,z) = (2,0,0)$ or $(2,1,1)$. \\
\end{theorem}

\noindent
Let us note that a generalization of this theorem to skew-symmetrizable matrices is not immediate because the Markov constant is not defined for non-skew-symmetric matrices; it is also not defined for skew-symmetric matrices whose diagrams are acyclic. It was observed in \cite{S5} that, for a skew-symmetric matrix $B$ of size $3$ and an admissible quasi-Cartan companion $A$ of $B$, we have $detA=2(4-C(B))$, leading to Theorem~\ref{th:acyclic}. For skew-symmetrizable matrices of arbitrary size, it seems that one needs to consider the admissible quasi-Cartan companion itself rather than just its determinant, see \cite{S3} for a conjecture. 
%The Markov constant appeared earlier in the literature, particularly on vector bundles, see e.g. \cite{Rudakov}. The relation between these contexts has not been particularly studied.

\section{Proofs of Main Results}
\label{sec:proof}
First we will prove some lemmas that we use to prove our theorems. In the proofs we assume, without loss of generality, that all diagrams are connected. The following statement generalizes \cite[Lemma~2.1c]{BBH}.

%Let us also note that we will usually consider equality of diagrams upto an enumeration of vertices.???
%\begin{theorem}\label{th:minimum}
%Suppose that $B$ is a skew-symmetrizable matrix of size $3$. Then the following holds:
%Suppose that $M$ is a mutation class of diagrams with $3$ vertices. For any $\Gamma$ in $M$, let $s(\Gamma)$ denote the sum of (square roots of) the weights of its edges. Then $M$ contains a diagram $\Gamma_0$ such that $s(\Gamma_0)$ is minimal. The diagram $\Gamma_0$ is unique if it is mutation-cyclic; otherwise it is acyclic and unique up to a change of orientation.

%Furthermore, for each $\Gamma$ in $M$, there is a sequence $\mu_i$ of mutations such that $\Gamma_0=\mu_1...\mu_n(\Gamma)$ and for $\Gamma_i=\mu_i...\mu_1(\Gamma)$ we have $s(\Gamma_i)<s(\Gamma_{i+1})$. 

%\end{theorem}

%RESTATE:
%\begin{theorem}\label{th:minimum}
%Suppose that $B$ is a skew-symmetrizable matrix of size $3$. Then the following holds:
%LEMMA: If  $\Gamma$ is mutation cyclic and $s(\Gamma)$ is not minimal then there is uniqu edge with maximal weight.????

%LEMMA: Suppose that $\Gamma=(a,b,c)$ is cyclic with $a<c$. Let $i$ be the vertex opposite to $c$ and $j$ the vertex opposite to $a$. 
%Suppose also that each $\mu_i(\Gamma)$ and $\mu_j(\Gamma)$ are cyclic.
%Then $s(\mu_i(\Gamma))< s(\mu_j(\Gamma))$.

%Proof. Note that $s(\mu_i(\Gamma))=a+b+(ab-c)$ and $s(\mu_j(\Gamma))=c+b+(bc-a)$. It canbe checked easily that the claimed inequality is satisfied.

\begin{lemma}\label{lem:BBH-Lemma2.1c}
Suppose that $\Gamma$ is a three-vertex diagram. If there are vertices $i\ne j$ such that $s(\mu_i(\Gamma)) < s(\Gamma)$ and $s(\mu_j(\Gamma))\leq s(\Gamma)$, then $\Gamma$ has an edge whose weight is less than $4$ and it is mutation-acyclic.

\end{lemma}

\begin{proof}
 Let us first note that $\Gamma$ is not acyclic (otherwise for any vertex $k$, we have $s(\mu_k(\Gamma))\geq s(\Gamma)$ ), so we suppose that $\Gamma=(a,b,c)$ is  cyclic. We assume that $i$ is not incident to the edge with radical weight $a$ and $j$ is not incident to the edge with radical weight $b$. Let us first assume that $\mu_i(\Gamma)$ and $\mu_j(\Gamma)$ are both cyclic.  Then $\mu_i(\Gamma)=(-a+bc,b,c)$ and $\mu_j(\Gamma)=(a,-b+ac,c)$. The conditions of the lemma imply that $bc < 2a$ and $ac\leq 2b$ (Proposition~\ref{prop:weight-s}(ii)). Multiplying the first inequality by $c$, we have $bc^2 < 2ac$; on the other hand, by the second inequality, we have $2ac\leq 4b$, implying that $c^2<4$. Then by Proposition~\ref{prop:weight 4}, the diagram $\Gamma$ is mutation-acyclic.

%Let us first note that $\Gamma$ is not acyclic (otherwise for any vertex $k$, we have $s(\mu_k(\Gamma))\geq s(\Gamma)$ ), so we suppose that $\Gamma=(alpha,beta,gamma)$ is  cyclic. We assume that $i$ is opposite to $\alpha$ and $j$ is opposite to $\beta$. Also assume first that $\mu_i(\Gamma)$ and $\mu_j(\Gamma)$ are both cyclic.  Then $\mu_i(\Gamma)=-\alpha+\beta\gamma,\beta,\gamma$ and $\mu_j(\Gamma)=\alpha,-\beta+\alpha\gamma,\gamma$. The conditions of the lemma imply that $\beta\gamma < 2\alpha$ and $\alpha\gamma\leq 2\beta$. Multiplying the first inequality by $\gamma$, we have $\beta\gamma^2 < 2\alpha\gamma$; on the other hand, by the second inequality, we have $2\alpha\gamma\leq 4\beta$, so $\gamma^2<4$ so $\gamma < 2(=4)$. Then by Proposition~\ref{prop:weight 4}, the diagram $\Gamma$ is mutation-acyclic.

Let us now assume that one of $\mu_i(\Gamma)$ or $\mu_j(\Gamma)$ is acyclic; without loss of generality, suppose that $\mu_i(\Gamma)$ is acyclic. Then $a\geq bc$. If $\mu_j(\Gamma)$ is not acyclic, then by the condition of the lemma we have $ac\leq 2b$. Then $ac^2\leq 2bc\leq 2a$, implying that $c^2\leq 2$. Similarly if $\mu_j(\Gamma)$ is acyclic, then $ac \leq b$, implying $a\leq ac^2\leq bc\leq a$, so $c=1$ (and $a=b$). In any case, by Proposition~\ref{prop:weight 4}, the diagram $\Gamma$ is mutation-acyclic. This completes the proof.
\end{proof}
%and $\mu_i(\Gamma)=(a-bc,b,c)$...

In view of the previos lemma, the following is a special case of Theorem~\ref{th:minimum}(ii):

\begin{lemma}\label{lem:weight 4-minimal}

Suppose $\Gamma$ be a three-vertex diagram which has an edge whose weight is less than $4$. If $\Gamma$ is cyclic, then it has a vertex $k$ such that $s(\Gamma)>s(\mu_k(\Gamma))$; in particular, $s(\Gamma)$ is not minimal.

\end{lemma}

\begin{proof}
 Suppose that $\Gamma=(a,b,c)$ with $a^2<4$ (so $a=1,\sqrt{2},\sqrt{3}$). If the conclusion of the lemma is not satisfied then we have $bc\geq 2a$, $ac\geq 2b$ and $ab\geq 2c$ (here $\mu_i(\Gamma)$ is cyclic for any vertex $i$ in $\Gamma$). Note that from the last inequality we have $b\geq 2c/a$, on the other hand, from the second inequality, we have $c\geq 2b/a$, implying $b\geq 4b/a^2$; however this is not possible because $a^2<4$. This completes the proof.
\end{proof}

Let us now show the following special case of Theorem~\ref{th:minimum2}(ii):

\begin{lemma}\label{lem:mut-cyc-non-min}

Suppose that $\Gamma$ is mutation-cyclic and has a vertex $i$ such that $s(\mu_i(\Gamma)) < s(\Gamma)$ (in particular $s(\Gamma)$ is not minimal). Then $i$ is the unique vertex with this property. Furthermore $\Gamma$ has a unique edge $e$ with maximal weight: the vertex $i$ is the vertex which is not incident to $e$.

\end{lemma}

\noindent
Note that the statement may not be true if $\Gamma$ is mutation-acyclic (e.g. for $\Gamma=(1,3,5)$).

 %We prove both parts together. Since $s(\Gamma)$ is not minimal there is a vertex $i$ such that $s(\mu_i(\Gamma)) < s(\Gamma)$. 
\begin{proof}
%Let $j,k$ be the remaining vertices. 
%Since $\Gamma$ is mutation-cyclic, by Lemma\ref{lem:BBH-Lemma2.1c}, we have $s(\mu_j(\Gamma)), s(\mu_k(\Gamma)) > s(\Gamma)$, which proves the first part of the statement. 
The first part of the statement follows from Lemma~\ref{lem:BBH-Lemma2.1c}. For the second part, suppose $\Gamma=(a,b,c)$ and assume that $i$ is the vertex which is not incident to the edge with radical weight $c$. Then, by the first part, $ba<2c$ but $bc \geq 2a$ and $ac\geq 2b$. Multiplying the second inequality by $a$ we have $abc \geq 2a^2$;  since $2c^2>abc$, we have $2c^2>2a^2$, thus $c>a$. Similarly $c>b$. This completes the proof.
\end{proof}

\subsection{Proof of Theorems~\ref{th:minimum} and ~\ref{th:minimum2}}
\label{subsec:proof-minimum2}
We will first prove Theorem~\ref{th:minimum} and Theorem~\ref{th:minimum2}(i) both at the same time. For this, let us first note that the mutation class of $\Gamma$ obviously contains a diagram $\Gamma_0$ such that $s(\Gamma_0)$ is minimal. To show its uniqueness, note that if $s(\Gamma)$ is minimal, then it satisfies:

(*) $s(\Gamma)\leq s(\mu_i(\Gamma))$ for any vertex $i$. 

\noindent
We will show that, in the mutation class of $\Gamma$, there is a unique, upto a change of orientation as described in the statement of Theorem~\ref{th:minimum}, diagram $\Gamma_0$ that satisfies (*). (This will prove Theorem~\ref{th:minimum} and Theorem~\ref{th:minimum2}(i)). 
For this purpose, %In other words, we first prove part (iii).???
let us suppose that $\Gamma'_0$ is another diagram that satisfies (*) in the mutation class of $\Gamma$; say $\Gamma'_0=\mu_n...\mu_1(\Gamma_0)$. We may assume without loss of generality that $n$ is minimal (in particular $i\ne i+1$, i.e. a mutation is not applied consecutively). If $n=1$, then, the effect of $\mu_1$ on $\Gamma_0$ is to reverse all edges (because $\Gamma_0$ and $\Gamma'_0=\mu_1(\Gamma_0)$ have the same weights as they satisfy (*)), so $\Gamma_0$ and $\Gamma'_0$ are equal as claimed in Theorem~\ref{th:minimum}. Thus for the rest of the proof, we can assume $n\geq 2$. 

To proceed, let us denote $\Gamma_l=\mu_l...\mu_1(\Gamma_0)$, $l=1,2,...,n$, with $\Gamma_n=\Gamma'_0$. Note that, since $n$ is minimal, $\Gamma_i$ $i=1,2,...,n-1$, does not satisfy (*)
; also we have
\begin{align} 
\label{eq:(**)}
&\text{$s(\Gamma_0)<s(\Gamma_1)$ and $s(\Gamma'_{0})< s(\Gamma_{n-1})$ }
\end{align}
%(**)$s(\Gamma_0)<s(\Gamma_1)$ and $s(\Gamma'_{0})< s(\Gamma_{n-1})$ 
(otherwise $\Gamma_1$ or $\Gamma_{n-1}$ satisfies (*) by Proposition~\ref{prop:weight-s}, contradicting the minimality of $n$; note that $\Gamma_{n-1}=\mu_n(\Gamma'_0)$). Let us also note that there exists $1\leq m \leq n-1$ such that $s(\Gamma_m)$ is maximum, i.e. $s(\Gamma_i)\leq s(\Gamma_m)$ for $i=0,1,...,n$. By \eqref{eq:(**)}, we can assume that $s(\Gamma_{m-1})\leq s(\Gamma_{m})>s(\Gamma_{m+1})$ (note that $\Gamma_{m-1}=\mu_{m}(\Gamma_{m})$ and $\Gamma_{m+1}=\mu_{m+1}(\Gamma_{m})$). Then, by Lemma~\ref{lem:BBH-Lemma2.1c}, the diagram $\Gamma$ is mutation-acyclic, so we are done if $\Gamma$ is mutation-cyclic. 

For the rest of the proof we assume that $\Gamma$ is mutation-acyclic. Note then that, by Lemma~\ref{lem:weight 4-minimal}, the diagrams $\Gamma_0$ and $\Gamma'_0$ are acyclic. Also note that $\Gamma_i$, $1\leq i \leq n-1$, is not acyclic because any acyclic diagram satisfies (*)
; in particular, $\Gamma_1$ is obtained from $\Gamma_0$ by mutating at the vertex which is not a source nor a sink (similarly the vertex $n$ is neither a source nor a sink in $\Gamma'_0=\Gamma_n$). Let us also note that, by Lemma~\ref{lem:BBH-Lemma2.1c}, the diagram $\Gamma_m$ has an edge whose weight is less than $4$, then it follows from Proposition~\ref{prop:weight-s} that 
\begin{align} 
\label{eq:(***)}
&\text{the diagram $\Gamma_l$ has an edge whose weight is less than $4$, for $l=0,1,2,...,n$.}
\end{align}
\credit{Case 1} \emph{$\Gamma_0$ is skew-symmetric, so has an edge of weight one in $\Gamma_0$.} Let us note that this case is known by a result of Caldero-Keller \cite[Conjecture~4.14 (4)]{FZ-CDM03}. We include a proof here to illustrate our method (which is very different from the one in \cite{FZ-CDM03}). 

For convenience, we first consider the case where $\Gamma_0$ is a tree, say $\Gamma_0=(a,1,0)^-$. Then $\Gamma_1=\mu_1(\Gamma_0)=(a,a,1)$ (recall that the vertex $1$ is neither a source nor a sink in $\Gamma_0$ because $s(\Gamma_0)<s(\Gamma_1)$). Then the diagram $\Gamma_2$ is obtained from $\Gamma_1$ by mutating at a vertex $k$ which is different from the vertex $1$. For such a vertex $k\ne 1$, we have the following: either $\Gamma_2=\mu_k(\Gamma_1)$ is equal to $\Gamma_0$ upto an enumeration of vertices, so the conclusion of the theorem holds, or $\Gamma_2=\mu_k(\Gamma_1)=(a,a,a^2-1)$. 
Let us note that, in the case where $\Gamma_2=\mu_k(\Gamma_1)=(a,a,a^2-1)$, we have the following: if $a>1$, then $\Gamma_2=\mu_k(\Gamma_1)$ does not have any edge whose weight is less than $4$, contradicting \eqref{eq:(***)} (here $a\geq 2$ because $a$ is an integer in this case); if $a=1$, then $\Gamma_2=\mu_k(\Gamma_1)$ is equal to $\Gamma_0$ upto an enumeration of vertices as claimed.

%Then the mutation at any vertex $k\ne 1$ gives a tree which is equal to $\Gamma_0$ upto an enumeration of vertices, so the conclusion of the theorem holds.  

Suppose now that $\Gamma_0$ is not a tree, say $\Gamma_0=(1,a,b)^-$. Note that if the vertex $1$ is not incident to any edge whose weight is $1$, then $\Gamma_1=\mu_1(\Gamma_0)=(a,b,1+ab)$ with $a,b\geq 2$; contradicting our assumption that $\Gamma_1$ has an edge whose weight is less than $4$  \eqref{eq:(***)}.
Thus, in $\Gamma_0$, we can assume that the vertex $1$ is incident to an edge of weight $1$, the remaining radical weights are $a,b$ (so $\Gamma_0$ has a source or sink whose incident edges have radical weights $a,b$). 

%It follows from a direct check that, for $i=1,...,n-1$, 
%$\Gamma_i=(1,b,a+b)$ or $\Gamma_i=(1,a,a+b)$. Then it can be checked easily that $\Gamma'_0$ has the same weights as $\Gamma_0$ and they can be obtained from each other with an enumeration of vertices and/or a 
%(???sequence of??? single reflection is enough) 
%reflection at a source or sink. 

To proceed, let us first consider the subcase where $a$ or $b$ is equal to one; without loss of generality, say $a=1$. Then it can be checked easily that, for $i=1,...,n-1$, %(here, in fact, $n=3$), 
we have $\Gamma_i=(1,1,a+1)$ or $\Gamma_i=(1,a,a+1)$ (so $\Gamma_n=\Gamma'_0=(1,1,a)^-$) and $\Gamma_0$ is equal to $\Gamma'_0$ upto an enumeration of vertices possibly after a reflection at a source or sink. 

Let us now assume that $a$ and $b$ are greater than one. We may assume without loss of generality that (in $\Gamma_0$) the vertex $1$ is not incident to the edge with radical weight $a$ (otherwise it is incident to the edge with radical weight $b$, then we exchange the letters $a$ and $b$).
Let us also assume that $\Gamma_0=(1,a,b)^-$ such that the edge 
$2 \longrightarrow 1$ has radical weight $1$, the edge $1\longrightarrow 3$ has radical weight $b$ and the edge $2 \longrightarrow 3$ has radical weight $a$.
Then $\Gamma_1=\mu_1(\Gamma_0)=(1,a+b,b)$. %Let us denote by $e$ the edge with weight one (the remaining radical weights are $a+b,b$). 
Now, the diagram $\Gamma_2$ is obtained by mutating $\Gamma_1$ at a vertex $v\ne 1$ such that $v$ is incident to an edge $e$ with weight one (otherwise $\mu_v(\Gamma_1)$ does not have any edge whose weight is less than $4$); so $e$ is the edge $1 \longrightarrow 2$ in $\Gamma_1$, thus $v=2$.
%(if $b=1$ and $v$ is incident to the edge $3\longrightarrow 1$ with radical weight $b$, then we may exchange the vertices $2,3$ by renaming). 
Then $\Gamma_2=\mu_2(\Gamma_1)=(1,a+b,a)$. Similarly the diagram $\Gamma_3$ is obtained by mutating $\Gamma_2$ at a vertex $w \ne 2$ such that $w$ is incident to an edge with weight one, then $w=1$ (because the vertex $3$ is incident to the edges with radical weights $a+b,a$, which are greater than or equal to $2$). Then $\Gamma_3=\mu_1(\Gamma_1)=(1,b,a)^-$ where the edge $1 \longrightarrow 2$ has radical weight $1$, the edge $3\longrightarrow 1$ has radical weight $a$ and the edge $3 \longrightarrow 2$ has radical weight $b$. Thus $\Gamma'_0=\Gamma_3=(1,b,a)^-$. Note then that $\Gamma'_0$ can be obtained from $\Gamma_0$ by first applying the reflection at the vertex $3$ then exchanging (renumbering) the vertices $1$ and $2$. This completes the case.

%To proceed, let us also assume that (in $\Gamma_0$) the vertex $1$ is not incident to the edge with radical weight $a$ (the case where the vertex $1$ is not incident to the edge with radical weight $b$ is considered similarly). 

%???We may assume wlog that $a=1$. Then it follows from a direct check that, for $i=1,...,n-1$, $\Gamma_i=(1,b,1+b)$ or $\Gamma_i=(1,1,1+b)$. Mutating at a vertex ($n$) such that $\Gamma'_0=\mu_n(\Gamma_{n-1})$ is acyclic, it can be checked easily that $\Gamma'_0$ has the same weights as $\Gamma_0$ and the conlusions hold.
%$\Gamma'_0=\mu_n(\Gamma_{n-1})=(a,b,1+ab)$
%\vfil

%\noindent
\credit{Case 2} \emph{$\Gamma_0$ is not skew-symmetric and has an edge of weight one in $\Gamma_0$}. Let us denote this edge by $e$. 
As in the previous case, let us first suppose that $\Gamma_0=(a,1,0)^-$ is a tree. Then $\Gamma_1=\mu_1(\Gamma_0)=(a,a,1)$ (recall that $1$ is the vertex which is neither a source nor a sink in $\Gamma_0$). For convenience, let us first assume that $a\geq 2$. Then, 
for any vertex $k\ne 1$, either $\Gamma_2=\mu_k(\Gamma_1)$ is equal to $\Gamma_0$ upto an enumeration of vertices, so the conclusion of the theorem holds, or $\Gamma_2=\mu_k(\Gamma_1)=(a,a,a^2-1)$, which does not have any edge whose weight is less than $4$, contradicting \eqref{eq:(***)}. Let us now assume that $a<2$, so $a=\sqrt{2}$ or $a=\sqrt{3}$. If $a=\sqrt{2}$, then for any vertex $k\ne 1$ either $\Gamma_2=\mu_k(\Gamma_1)$ is equal to $\Gamma_0$ upto an enumeration of vertices (so the conclusion of the theorem holds) or $\Gamma_2=\mu_k(\Gamma_1)=(\sqrt{2},\sqrt{2},1)$. Similarly, in the case where $\Gamma_2=\mu_k(\Gamma_1)=(\sqrt{2},\sqrt{2},1)$, we have $\Gamma_3=\mu_j(\Gamma_2)$ for some $j\ne k$; furthermore $\Gamma_3$ is equal to $\Gamma_0$ or it is equal to $\Gamma_2$ upto an enumeration of vertices. Continuing by induction, we have the following conclusion (when $a=\sqrt{2}$): for all $i=1,...n-1$, $\Gamma_i=(\sqrt{2},\sqrt{2},1)$ and $\Gamma'_0$ is equal to $\Gamma_0$ upto an enumeration of vertices. 
If $a=\sqrt{3}$, by a similar argument, we have the following: for all $i=1,...n-1$, $\Gamma_i=(\sqrt{3},\sqrt{3},1)$ or $\Gamma_i=(\sqrt{3},\sqrt{3},2)$, and $\Gamma'_0$ is equal to $\Gamma_0$ upto an enumeration of vertices. 

%Then the diagram $\Gamma_2$ is obtained from $\Gamma_1$ by mutating at a vertex $k$ which is different from the vertex $1$. We note that, for any vertex $k\ne 1$ either $\Gamma_2=\mu_k(\Gamma_1)$ is equal to $\Gamma_0$ upto an enumeration of vertices, so the conclusion of the theorem holds; or $\Gamma_2=\mu_k(\Gamma_1)=(a,a,a^2-1)$. 

%Then mutation at any vertex $k\ne 1$ gives a tree with the same weights, so the uniqueness conclusion of the theorems holds.  

%Then the diagram $\Gamma_2$ is obtained from $\Gamma_1$ by mutating at a vertex $k$ which is different from the vertex $1$. Let us now note that, for any vertex $k\ne 1$, we have the following: either $\Gamma_2=\mu_k(\Gamma_1)$ is equal to $\Gamma_0$ upto an enumeration of vertices, so the conclusion of the theorem holds; or $\Gamma_2=\mu_k(\Gamma_1)=(a,a,a^2-1)$. 
%Let us note that, in the case where $\Gamma_2=\mu_k(\Gamma_1)=(a,a,a^2-1)$, we have the following: if $a>1$, then $\Gamma_2=\mu_k(\Gamma_1)$ does not have any edge whose weight is less than $4$, contradicting our assumption \eqref{eq:(***)} (here $a\geq 2$ because $a$ is an integer in this case); if $a=1$, then $\Gamma_2=\mu_k(\Gamma_1)$ is equal to $\Gamma_0$ upto an enumeration of vertices as claimed.

Let us now suppose that $\Gamma_0=(1,a,b)^-$ is not a tree. Note that since $\Gamma_0$ is not skew-symmetric, the numbers $a,b$ are not integers (but square-roots of integers). For convenience, we consider in subcases:

%Let $k(=1)$ is the vertex which is neither source nor sink in $\Gamma_0$. (Assume first that $k$ is not incident to $e$ NOT NECESSARY!!!  below may need to assume that $a=\sqrt{2}$ or $a=\sqrt{3}$, which is done, but need to do remaining possibility $a,b\geq 4$ seperately). 

%Let $a,b$ be the weights of the remaining edges, where $a,b\geq \sqrt{2}$ (otherwise finite type)
%\noindent
\credit{Subcase 2.1} \emph{$a$ or $b$ is equal to $\sqrt{2}$.}
%Let $a,b$ be the weights of the remaining edges, where $a,b\geq \sqrt{2}$ (otherwise finite type). Note that $\mu_k(\Gamma_0)$ has weights $a,b,ab+1$, thus $a$ or $b< 2$; 
%assume w.l.o.g $a=\sqrt{2}$ or $a=\sqrt{3}$. 
Let us assume, without loss of generality, that $a=\sqrt{2}$. (Note then that $b=m\sqrt{2}$ where $m$ is integer). Then, by similar arguments as in Case 1 above, 
%it follows from a direct check 
it follows that $\Gamma_i$, $1\leq i \leq n-1$, (in fact $n \leq 5$), belongs to one of the following types (in the notation of Definition~\ref{def:(a,b,c)}): 
$(\sqrt{2},b,\sqrt{2}b+1)$; $(\sqrt{2},b+\sqrt{2},\sqrt{2}b+1)$; $(\sqrt{2}+b,\sqrt{2},1)$; $(\sqrt{2}+b,1,b)$ 
such that $\Gamma'_0$ can be obtained from $\Gamma_0$ possibly after enumerating the vertices and reflecting at a source or sink. 

%\noindent
%More explicitly, for any $\Gamma_i$ satisfying (**) and a vertex $k$ in $\Gamma_i$, the diagram $\mu_k(\Gamma_i)$ is either of one of these types or has an edge weight greater than $4$ or is an acyclic diagram $\Gamma'_0$; furthermore $\Gamma'_0$ can be obtained from $\Gamma_0$ by a source sink reflection.

%\noindent
\credit{Subcase 2.2} \emph{$a$ or $b$ is equal to $\sqrt{3}$.} Let us assume, without loss of generality ,that $a=\sqrt{3}$: Then, by similar arguments as in Case 1 above, it follows that $\Gamma_i$, $1\leq i \leq n-1$, (in fact $n\leq 6$), is of one of the following types: $(\sqrt{3},b,1+\sqrt{3}b)$; $(\sqrt{3},\sqrt{3}+2b,1+\sqrt{3}b)$; 
$(\sqrt{3},\sqrt{3}+2b,2+\sqrt{3}b)$; $(\sqrt{3},\sqrt{3}+b,2+\sqrt{3}b)$; $(\sqrt{3},\sqrt{3}+b,1)$; $(b,\sqrt{3}+b,1)$ such that $\Gamma'_0$ can be obtained from $\Gamma_0$ possibly after renumbering the vertices and reflecting at a source or sink.   

%\noindent
%More explicitly, for any $\Gamma_i$ satisfying (***) and a vertex $k$ in $\Gamma_i$, the mutation gives a diagram which is either of one of these types or has any edgeweight greater than $4$ or acyclic $\Gamma'_0$; furthermore $\Gamma'_0$ can be obtained from $\Gamma_0$ by a source sink reflection.

%???If $a,b\geq 4$ then $\Gamma_i$, $1<i<n$, is $a,b,1$;$a,a+b,1$;$b,a+b,1$???. Then it follows from a direct check that $\Gamma'_0$ has the same weights as $\Gamma_0$.

%\noindent
\credit{Subcase 2.3} \emph{$a,b\geq 2$.} Note that if the vertex $1$ is not incident to $e$, then $\Gamma_1=\mu_1(\Gamma_0)=(a,b,ab+1)$, so $\Gamma_1$ does not have any edge whose weight is less than four, contradicting our assumption. Thus in this case the vertex $1$ is incident to $e$. Then $\Gamma_i$, $1\leq i\leq n-1$, (in fact $n=3$), is of type $(1,a+b,b)$ or $(1,a+b,a)$ such that $\Gamma'_0$ can be obtained from $\Gamma_0$ by a reflection at a source or sink and renumbering the vertices if necessary.

%\noindent
\credit{Case 3} \emph{$\Gamma_0$ is not skew-symmetric and minimal edge weight is equal to ${2}$.} Let us write $\Gamma_0=(\sqrt{2},a,b)^-$ , where $a,b\geq \sqrt{2}$; if $\Gamma_0$ is a tree, then we take $b=0$.
%$\sqrt{2}$.}
%Let $a,b$ be the radical weights of the remaining edges, where $a,b\geq \sqrt{2}$ ; if $\Gamma_0$ is a tree then $a$ denotes the "weight" of the only remainging edge.

%???if $\Gamma_0$ is a tree, then it ($\Gamma_i$) has weights $\sqrt{2},a,\sqrt{2}a$; thus we may assume that $\Gamma_0$ is cycle (triangle).!!! (CONSIDER THIS IN THE SUBCASES)

%\noindent
\credit{Subcase 3.1} \emph{$a$ or $b$ is equal to $\sqrt{2}$.} Let us assume, without loss of generality, that $a=\sqrt{2}$. 
If $\Gamma_0=(\sqrt{2},\sqrt{2},0)^-$ is a tree, then it is easily checked (under the assumption \eqref{eq:(***)}) that $\Gamma_i=(\sqrt{2},\sqrt{2},2)$, $i=1,...n-1$, and $\Gamma'_0$ is as required in the conclusion of the uniqueness claims in the theorems. Let us now assume that $\Gamma_0$ is not a tree. (Note then that $b$ is an integer). Then $\Gamma_i$, $1\leq i \leq n-1$,  (in fact $n=4$), belongs to one of the following types: %$(\sqrt{2},a+\sqrt{2}b,b)$; $(\sqrt{2},a+\sqrt{2}b,b+\sqrt{2}a)$; $(\sqrt{2},\sqrt{2},b+{2})$. 
$(\sqrt{2},\sqrt{2}+\sqrt{2}b,b)$; $(\sqrt{2},\sqrt{2}+\sqrt{2}b,b+{2})$; $(\sqrt{2},\sqrt{2},b+{2})$ such that $\Gamma'_0$ can be obtained from $\Gamma_0$ by a reflection at a source or sink and renumbering the vertices if necessary.

%\noindent
\credit{Subcase 3.2} \emph{$a$ or $b$ is equal to $\sqrt{3}$.} Let us assume, without loss of generality, that $a=\sqrt{3}$. If $\Gamma_0$ is a tree, then it is easily checked that, for $1\leq i\leq n-1$, $\Gamma_i=(\sqrt{2},\sqrt{6},\sqrt{3})$; $(2\sqrt{2},\sqrt{6},\sqrt{3})$ and the uniqueness conlusion of the theorems is satisfied.
Let us now assume that $\Gamma_0$ is not a tree. (Note then that $b=\sqrt{2}\sqrt{3}m$ where $m$ is integer).
Then $\Gamma_i$, for $1\leq i\leq n-1$, (in fact $n\leq 6$), is of one of the following types: $(\sqrt{3},\sqrt{2}+\sqrt{3}b,b)$; $(\sqrt{3},\sqrt{2}+\sqrt{3}b,\sqrt{6}+2b)$; $(\sqrt{3},2\sqrt{2}+b\sqrt{3},\sqrt{6}+2b)$; $(\sqrt{3},2\sqrt{2}+b\sqrt{3},\sqrt{6}+b)$; $(\sqrt{3},\sqrt{2},\sqrt{6}+b)$ ; $(\sqrt{2},\sqrt{3}+\sqrt{2}b,b)$; $(\sqrt{2},\sqrt{3}+\sqrt{2}b,\sqrt{6}+b)$ such that $\Gamma'_0$ can be obtained from $\Gamma_0$ by a reflection at a source or sink and renumbering the vertices if necessary.

%\noindent
\credit{Subcase 3.3} \emph{$a,b \geq {2}$.} %(Make this last subcase)
%Assume first that the remaining edge weights $a,b\geq 4$. 
 Note that if the vertex $1$ is not incident to the edge with radical weight $\sqrt{2}$, because otherwise $\Gamma_1=\mu_1(\Gamma_0)=(\sqrt{2}+ab,a,b)$ does not have any edge whose weight is less than four, contradicting our assumption. Then, by similar arguments as in Case 1, it follows that $\Gamma_i$, $1\leq i\leq n-1$, (in fact $n=4$),
belongs to one of the following types: 
$(\sqrt{2},a,b+\sqrt{2}a)$; $(\sqrt{2},a+\sqrt{2}b,b+\sqrt{2}a)$; $(\sqrt{2},a+\sqrt{2}b,b)$ such that $\Gamma'_0$ can be obtained from $\Gamma_0$ by a reflection at a source or sink and renumbering the vertices if necessary.

%\noindent
\credit{Case 4} \emph{$\Gamma_0$ is not skew-symmetric and minimal edge weight is equal to ${3}$.}
Let us write $\Gamma_0=(\sqrt{3},a,b)^-$ , where $a,b\geq \sqrt{3}$; if $\Gamma_0$ is a tree, then we take $b=0$.
%Let $a,b$ be the weights of the remaining edges, where $a,b\geq \sqrt{3}$; if $\Gamma_0$ is a tree then $a$ denotes the "weight" of the only remainging edge.
If $\Gamma_0$ is a tree, then it is easily checked (under the assumption \eqref{eq:(***)}) that,
for $1\leq i\leq n-1$, (in fact, $n=5$), $\Gamma_i=(\sqrt{3},a,\sqrt{3}a)$ or $\Gamma_i=(\sqrt{3},2a,\sqrt{3}a)$ and the uniqueness conclusion of the theorems is satisfied.
We now assume that $\Gamma_0$ is a cycle (triangle). Suppose first that one of the radical weights $a,b$ is less than two; without loss of generality, say $a=\sqrt{3}$. Then, by similar arguments as in Case 1, it follows that, for $1\leq i\leq n-1$, (in fact, $n=6$), the diagram $\Gamma_i$ belongs to one of the following types: $(\sqrt{3},\sqrt{3}+\sqrt{3}b,b)$; $(\sqrt{3},\sqrt{3}+\sqrt{3}b,3+2b)$; $(\sqrt{3},2\sqrt{3}+\sqrt{3}b,3+2b)$; $(\sqrt{3},2\sqrt{3}+\sqrt{3}b,3+b)$; $(\sqrt{3},\sqrt{3},3+b)$ such that $\Gamma'_0$ can be obtained from $\Gamma_0$ by a reflection at a source or sink and renumbering the vertices if necessary.

Suppose now that $a,b\geq 2$. Note that the vertex $1$ is not incident to the edge with radical weight $\sqrt{3}$, because otherwise $\Gamma_1=\mu_1(\Gamma_0)=(\sqrt{3}+ab,a,b)$ does not have any edge whose weight is less than four, contradicting our assumption.
Then, by similar arguments as in Case 1, it follows that $\Gamma_i$, $1\leq i\leq n-1$, (in fact, $n=6$),
belongs to one of the following types: 
$(\sqrt{3},a,b+\sqrt{3}a)$; $(\sqrt{3},2a+\sqrt{3}b,b+\sqrt{3}a)$; $(\sqrt{3},2a+\sqrt{3}b,2b+\sqrt{3}a)$; $(\sqrt{3},a+\sqrt{3}b,2b+\sqrt{3}a)$;$(\sqrt{3},a+\sqrt{3}b,b)$ such that $\Gamma'_0$ can be obtained from $\Gamma_0$ by a reflection at a source or sink and renumbering the vertices if necessary. This completes the case.

We have completed the proof of Theorem~\ref{th:minimum} and Theorem~\ref{th:minimum2}(i). Then Theorem~\ref{th:minimum2}(ii) follows from these statements and Lemma~\ref{lem:mut-cyc-non-min}. This completes the proofs of the theorems.

\subsection{Proof of Theorem~\ref{th:gcd}}
\label{subsec:proof-gcd}
It is enough to show the theorem for $\Gamma'=\mu_k(\Gamma)$, where $k$ is a vertex in $\Gamma$. In the proof we will use the following notation: for any two vertices $i$ and $j$, we denote by $\omega_{i,j}$ (resp. $\omega'_{i,j}$) the corresponding weight in $\Gamma$ (resp. $\Gamma'$) (note that $\omega_{i,j}=\omega_{j,i}$, also if the vertices $i,j$ are not connected in $\Gamma$, then $\omega_{i,j}=0$, similarly in $\Gamma'$). We will show that, for any vertex $i$, we have  $\delta_i(\Gamma')=\delta_i(\Gamma)$. For this purpose, let us first note that $\delta_k(\Gamma')=\delta_k(\Gamma)$ by the definition of the mutation (because the weights of the edges which are incident to $k$ are not affected). Similarly for any vertex $i$ which is not adjacent to $k$, we have $\delta_i(\Gamma')=\delta_i(\Gamma)$. To complete the proof, let us now assume that $i$ is a vertex which is adjacent to $k$. Then, by the description of the mutation of diagrams in Section~\ref{sec:pre}, 
for any vertex $j$, the weight $\omega'_{i,j}$ is equal to $\omega_{i,j}$ or $(\sqrt{\omega_{i,j}}\pm\sqrt{\omega_{j,k}\omega_{k,i}})^2=\omega_{i,j}\pm2\sqrt{\omega_{i,j}\omega_{j,k}\omega_{k,i}}+ \omega_{j,k}\omega_{k,i}$. 
%or $(\sqrt{\omega_{i,j}}-\sqrt{\omega_{j,k}\omega_{k,i}})^2=\omega_{i,j}-2\sqrt{\omega_{i,j}\omega_{j,k}w_{k,i}}+\omega_{j,k}\omega_{k,i}$. 
Here note that $\sqrt{\omega_{i,j}\omega_{j,k}\omega_{k,i}}$ is an integer by the definition of a diagram, furthermore it is divisible by $\delta_i$ (because $\delta_i$ divides both $\omega_{i,j}$ and $\omega_{k,i}$). Thus $\delta_i$ divides $\omega'_{i,j}$ for any $j$, so it divides $\delta'_i=\delta_i(\Gamma')$. Since $\mu_k$ is involutive,  $\delta'_i$ divides $\omega_{i,j}$ for any $j$, so $\delta'_i$ divides $\delta_i$ as well, consequently $\delta_i=\delta'_i$. (For diagrams of skew-symmetric matrices,  the same arguments work if $\delta_i$ is defined as the greatest common divisor of the radicals of the weights of the edges which are incident to the vertex $i$).
This completes the proof.  

%Let us give an example for a use of Theorem~\ref{th:gcd}:XXX

\end{document}